\documentclass[11pt]{amsart}

\newtheorem{theorem}{Theorem}[section]

\newtheorem{corollary}[theorem]{Corollary}

\newtheorem{definition}[theorem]{Definition}
\newtheorem{lemma}[theorem]{Lemma}

\newtheorem{proposition}[theorem]{Proposition}

\theoremstyle{definition}
\newtheorem{remark}[theorem]{Remark}
\usepackage[colorlinks, bookmarks=true]{hyperref}
\usepackage{color,graphicx,shortvrb}

\usepackage{enumerate}
\usepackage{amsmath}
\usepackage{amssymb}

\usepackage[latin 1]{inputenc}

\def\J#1#2#3{ \left\{ #1,#2,#3 \right\} }

\def\NN{{\mathbb{N}}}
\def\11{\textbf{$1$}}
\def\CC{{\mathbb{C}}}

\begin{document}

\numberwithin{equation}{section}

\title[Jordan weak amenability and orthogonal forms on JB$^*$-algebras]{Jordan weak amenability and orthogonal forms on JB$^*$-algebras}

\author[F.B. Jamjoom]{Fatmah B. Jamjoom}
\address{Department of Mathematics, College of Science, King Saud University, P.O.Box 2455-5, Riyadh-11451, Kingdom of Saudi Arabia.}
\email{fjamjoom@ksu.edu.sa}

\author[A.M. Peralta]{Antonio M. Peralta}
\address{Departamento de An{\'a}lisis Matem{\'a}tico, Universidad de Granada,\\
Facultad de Ciencias 18071, Granada, Spain}
\curraddr{Visiting Professor at Department of Mathematics, College of Science, King Saud University, P.O.Box 2455-5, Riyadh-11451, Kingdom of Saudi Arabia.}
\email{aperalta@ugr.es}

\author[A.A. Siddiqui]{Akhlaq A. Siddiqui}
\address{Department of Mathematics, College of Science, King Saud University, P.O.Box 2455-5, Riyadh-11451, Kingdom of Saudi Arabia.}
\email{asiddiqui@ksu.edu.sa}

\thanks{The authors extend their appreciation to the Deanship of Scientific Research at King Saud University for funding this work through research group no  RG-1435-020. Second author also partially supported by the Spanish Ministry of Science and Innovation, D.G.I. project no. MTM2011-23843.}

\begin{abstract} We prove the existence of a linear isometric correspondence between the Banach space of all symmetric orthogonal forms on a JB$^*$-algebra $\mathcal{J}$ and the Banach space of all purely Jordan generalized derivations from $\mathcal{J}$ into $\mathcal{J}^*$. We also establish the existence of a similar linear isometric correspondence between the Banach spaces of all anti-symmetric orthogonal forms on $\mathcal{J}$, and of all Lie Jordan derivations from $\mathcal{J}$ into $\mathcal{J}^*$.
\end{abstract}

\date{}
\maketitle

\section{Introduction}\label{sec:intro}

Let $\varphi$ and $\psi$ be functionals in the dual of a C$^*$-algebra $A$. The assignment $$(a,b) \mapsto V_{\varphi,\psi} (a,b) :=\varphi  \Big(\frac{a b+ ba}{2}\Big) +  \psi  \Big(\frac{a b- ba}{2}\Big)$$ defines a continuous bilinear form on $A$ which also satisfies the following property: given $a,b\in A$  with $a\perp b$ (i.e. $a b^* = b^* a=0$) we have $V_{\varphi,\psi} (a,b^*) = 0$.  A continuous bilinear form $V : A\times A \to \mathbb{C}$ is said to be \emph{orthogonal} when $V (a,b) = 0$ for every $a,b\in A_{sa}$ with $a\perp b$ (see \cite[Definition 1.1]{Gold}).  A renowned and useful theorem, due to S. Goldstein \cite{Gold}, gives the precise expression of every continuous bilinear orthogonal form on a C$^*$-algebra.

\begin{theorem}\label{thm Goldstein}\cite{Gold} Let $V : A\times A \to \CC$ be a continuous orthogonal form on a C$^*$-algebra.
Then there exist functionals $\varphi,\psi\in A^*$ satisfying that
$$ V (a,b) = V_{\varphi,\psi} (a,b)= \varphi (a \circ b) + \psi ([a,b]),$$ for
all $a,b\in A$, where $a\circ b:= \frac12 (ab +ba)$, and $[a,b]:=
\frac12 (ab -ba)$. $\hfill\Box$
\end{theorem}

Henceforth, the term ``form'' will mean a ``continuous bilinear form''. It should be noted here that by the above Goldstein's theorem, for every orthogonal form $V$ on a C$^*$-algebra we also have $V(a, b^*) = 0$, for every $a,b\in A$ with $a\perp b$.\smallskip

The applications of Goldstein's theorem appear in many different contexts (\cite{BurFerGarMarPe,BurFerGarPe,GarPePuRa,HaaLaust}). Quite recently, an extension of Goldstein's theorem for commutative real C$^*$-algebras has been published in \cite{GarPe14}.\smallskip

Making use of the weak amenability of every C$^*$-algebra, U. Haagerup and N.J. Laustsen gave a simplified proof of Goldstein's theorem in \cite{HaaLaust}. In the third section of the just quoted paper, and more concretely, in the proof of \cite[Proposition 3.5]{HaaLaust}, the above mentioned authors pointed out that for every anti-symmetric form  $V$ on a C$^*$-algebra $A$ which is orthogonal on $A_{sa}$, the mapping $D_{_V}: A \to A^*$, $D_{_V}(a) (b) = V(a,b)$ {\rm(}$a,b\in A${\rm)} is a derivation. Reciprocally, the weak amenability of $A$ also implies that every derivation $\delta$ from $A$ into $A^*$ is inner and hence of the form $\delta (a) = \hbox{adj}_{\phi} (a) = \phi a - a \phi$ for a functional $\phi \in A^*$. In particular, the form $V_{\delta} (a,b) = \delta (a) (b)$ is anti-symmetric and orthogonal.\smallskip

The above results are the starting point and motivation of the present note. In the setting of C$^*$-algebras we shall complete the above picture showing that symmetric orthogonal forms on a C$^*$-algebra $A$ are in bijective correspondence with the \emph{purely Jordan generalized derivations} from $A$ into $A^*$ (see Section \ref{sec:derivations} for definitions). However, the main goal of this note is to explore the orthogonal forms on a JB$^*$-algebra and the similarities and differences between the associative setting of C$^*$-algebras and the wider class of JB$^*$-algebras.\smallskip

In Section \ref{sec:derivations} we revisit the basic theory and results on Jordan modules and derivations from the associative derivations on C$^*$-algebras to Jordan derivations on C$^*$-algebras and JB$^*$-algebras. The novelties presented in this section include a new study about generalized Jordan derivations from a JB$^*$-algebra $\mathcal{J}$ into a Jordan Banach $\mathcal{J}$-module in the line explored in \cite{LiPan}, \cite[\S 4]{AlBreExVill09}, and \cite[\S 3]{BurFerPe2013}. We recall that, given a Jordan Banach $\mathcal{J}$-module $X$ over a JB$^*$-algebra, a \emph{generalized Jordan derivation} from $\mathcal{J}$ into $X$ is a linear mapping $G: \mathcal{J}\to X$ for which there exists $\xi\in X^{**}$ satisfying $$ G (a\circ b) = G(a)\circ b + a\circ G(b) - U_{a,b} (\xi ),$$ for every $a,b$ in $\mathcal{J}$. We show how the results on automatic continuity of Jordan derivations from a JB$^*$-algebra $\mathcal{J}$ into itself or into its dual, established by S. Hejazian, A. Niknam \cite{HejNik96} and B. Russo and the second author of this paper in \cite{PeRu}, can be applied to prove that every generalized Jordan derivation from $\mathcal{J}$ into $\mathcal{J}$ or into $\mathcal{J}^*$ is continuous (see Proposition \ref{p automatic cont generalized Jordan derivations}).\smallskip

Section \ref{sec:orthog forms} contains the main results of the paper. In Proposition \ref{prop generalized Jordan derivations define orthogonal forms on the whole JB-algebra} we prove that for every generalized Jordan derivation  $G: \mathcal{J} \to \mathcal{J}^*$, where $\mathcal{J}$ is a JB$^*$-algebra, the form $V_{G} : \mathcal{J}\times \mathcal{J} \to \mathbb{C}$, $V_{G} (a,b) = G(a) (b)$ is orthogonal on the whole $\mathcal{J}$. We introduce the two new subclasses of \emph{purely Jordan generalized derivations} and \emph{Lie Jordan derivations}. A generalized derivation $G: \mathcal{J} \to \mathcal{J}^*$ is said to be a purely Jordan generalized derivation if $G(a) (b) = G(b) (a)$, for every $a,b\in \mathcal{J}$; while a Lie Jordan derivation is a Jordan derivation $D: \mathcal{J} \to \mathcal{J}^*$ satisfying $D(a) (b) =- D(b) (a)$, for all $a,b\in \mathcal{J}$. \smallskip

Denoting by $\mathcal{OF}_{s}(\mathcal{J})$ the Banach space of all symmetric orthogonal forms on $\mathcal{J}$, and by $\mathcal{PJGD}er(\mathcal{J},\mathcal{J}^*)$ the Banach space of all purely Jordan generalized derivations from $\mathcal{J}$ into $\mathcal{J}^*$, then the mappings $$\mathcal{OF}_{s}(\mathcal{J}) \to \mathcal{PJGD}\hbox{er}(\mathcal{J},\mathcal{J}^*), \ \ \mathcal{PJGD}\hbox{er}(\mathcal{J},\mathcal{J}^*) \to \mathcal{OF}_{s}(\mathcal{J}),$$ $$V\mapsto G_{_V}, \ \ \ \  \ \ \ \  \ \ \ \  \ \ \ \  \ \ \ \  \ \ \ \  \ \ \ \ \ \ \ \  \ \ \ \ G\mapsto V_{_G},$$ define two isometric linear bijections and are inverses of each other (cf. Theorem \ref{t one-to-one correspondence between generalized purely Jordan derivations and symmetric orthognal forms}). When $\mathcal{OF}_{as}(\mathcal{J})$ and $\mathcal{L}ie\mathcal{JD}er(\mathcal{J},\mathcal{J}^*)$ denote the Banach spaces of all anti-symmetric orthogonal forms on $\mathcal{J}$, and of all Lie Jordan derivations from $\mathcal{J}$ into $\mathcal{J}^*$, respectively, the mappings
$$\mathcal{OF}_{as}(\mathcal{J}) \to \mathcal{L}\hbox{ie}\mathcal{JD}\hbox{er}(\mathcal{J},\mathcal{J}^*), \ \ \mathcal{L}\hbox{ie}\mathcal{JD}\hbox{er}(\mathcal{J},\mathcal{J}^*) \to \mathcal{OF}_{as}(\mathcal{J}),$$ $$V\mapsto D_{_V}, \ \ \ \  \ \ \ \  \ \ \ \  \ \ \ \  \ \ \ \  \ \ \ \  \ \ \ \ \ \ \ \  \ \ \ \ D\mapsto V_{_D},$$ define two isometric linear bijections and are inverses of each other (see Theorem \ref{t one-to-one correspondence between generalized purely Jordan derivations and anti-symmetric orthognal forms}).\smallskip

We culminate the paper with a short discussion which shows that contrary to what happens for anti-symmetric orthogonal forms on a C$^*$-algebra, the anti-symmetric orthogonal forms on a JB$^*$-algebra are not determined by the inner Jordan derivations from $\mathcal{J}$ into $\mathcal{J}^*$ (see Remark \ref{r there is no Jordan Goldstein theorem}). It seems unnecessary to remark the high impact and deep repercussion that the theory of derivations on C$^*$-algebras and JB$^*$-algebras, the results in this note add a new interest and application of Jordan derivations and generalized Jordan derivations on JB$^*$-algebras. \smallskip

Throughout this paper, we habitually consider a Banach space $X$ as a norm closed subspace of $X^{**}$. Given a  closed subspace $Y$ of $X$, we shall identify the weak$^*$-closure, in $X^{**}$, of $Y$ with $Y^{**}$.

\section{Derivations and generalized derivations in correspondence with orthogonal forms}\label{sec:derivations}

A {\it derivation} from a Banach algebra $A$ into a Banach
$A$-module $X$ is a linear map $D: A\to X$ satisfying $D(a b) = D(a) b +a  D(b),$ ($a\in A$). A {\it Jordan derivation} from $A$ into $X$
is a linear map $D$ satisfying $D(a^2) = a D(a) + D(a)
a,$ ($a\in A$), or equivalently, $D(a\circ b)=a\circ D(b)+ D(a)\circ b$ ($a,b\in A$), where $a\circ b = \frac{a b+b a}{2},$ whenever $a,b\in A$, or one from $\{a,b\}$ is in $A$ and the other is in $X$. Let $x$ be an element in a $X$,
the mapping $\hbox{adj}_{x} :A \to X$, $a\mapsto \hbox{adj}_{x} (a) := x a -  a x$, is an example of a derivation from $A$ into $X$. A derivation $D: A\to X$ is said to be an \emph{inner derivation} when it can be written in the form $D = \hbox{adj}_{x}$ for some $x\in X$.\smallskip

A well known result of S. Sakai (cf. \cite[Theorem 4.1.6]{Sak}) states that every derivation on a von Neumann algebra is inner.\smallskip

J.R. Ringrose proved in \cite{Ringrose72} that every derivation from a C$^*$-algebra $A$ into a Banach $A$-bimodule is continuous.\smallskip

A Banach algebra, $A$, is \emph{amenable} if every bounded derivation from $A$ into a dual Banach A-bimodule is inner. Contributions of A. Connes and U. Haagerup show that a C$^*$-algebra is amenable if and only if it is nuclear (\cite{Connes,Haa83}). The class of weakly amenable C$^*$-algebras is less restrictive.
A Banach algebra, $A$, is \emph{weakly amenable} if every bounded derivation from $A$ into $A^*$ is inner. U. Haagerup proved that every C$^*$-algebra $B$ is
weakly amenable, that is, for every derivation $D: B\to B^*$, there exists $\varphi\in B^*$ satisfying $D(.) = \hbox{adj}_{\varphi}$ (\cite[Corollary 4.2]{Haa83}).\smallskip

In \cite{LiPan} J. Li and Zh. Pan introduce a concept which generalizes the notion of derivation and is more related to the Jordan structure underlying a C$^*$-algebra. We recall that a linear mapping $G$ from a unital C$^*$-algebra $A$ to a (unital) Banach $A$-bimodule $X$ is called a \emph{generalized derivation} in \cite{LiPan} whenever the identity $$G (ab) = G(a) b + a G(b) - a G(1) b$$ holds for every $a,b$ in $A$. The non-unital case was studied in \cite[\S 4]{AlBreExVill09}, where a generalized derivation from a Banach algebra $A$ to a Banach $A$-bimodule $X$ is defined as a linear operator $D: A \to X$ for which there exists $\xi\in  X^{**}$ satisfying $$D(ab) = D(a)  b + a  D(b) - a \xi b \hbox{ ($a, b \in A$).}$$

Given an element $x$ in $X$, it is easy to see that the operator $G_{x} :A \to X$, $x\mapsto G_{x} (a):= a x + x a$, is a generalized derivation from $B$ into $X$. Since, in this setting, every derivation $D: A\to X$ satisfies $D(1) = 2 1 D(1)$ and hence $D(1) =0$, it follows that every derivation from $A$ into $X$ is a generalized derivation. There are examples of generalized derivations from a C$^*$-algebra $A$ into a Banach $A$-bimodule $X$ which are not derivations, for example $G_a : A \to A$ is a generalized derivation which is not a derivation when $a^* \neq -a$ (cf. \cite[comments after Lemma 3]{BurFerGarPe2014}).\smallskip

\subsection{Jordan algebras and modules}\label{subsec: Jordan modules}

We turn now our attention to Jordan structures and derivations. We recall that a  real (resp., complex) \emph{Jordan algebra} is a commutative algebra over the real (resp., complex) field which is not, in general associative, but satisfies the \emph{Jordan identity}:\begin{equation}\label{eq Jordan idenity algebra} (a \circ b)\circ a^2 = a\circ (b \circ a^2).
\end{equation} A normed Jordan algebra is a Jordan algebra $\mathcal{J}$ equipped with a norm, $\|.\|$, satisfying $\| a\circ b\| \leq \|a\| \ \|b\|$, $a,b\in \mathcal{J}$. A \emph{Jordan Banach algebra} is a normed Jordan algebra whose norm is complete. A JB$^*$-algebra is a complex Jordan Banach algebra $\mathcal{J}$ equipped with an isometric algebra involution $^*$ satisfying  $\|\J a{a^*}a \|= \|a\|^3$, $a\in \mathcal{J}$  (we recall that $\J a{a^*}a  =2 (a\circ a^*) \circ a - a^2 \circ a^*$). A real Jordan Banach algebra $\mathcal{J}$ satisfying $$\|a\|^2 = \|a^2\| \hbox{ and, } \|a^2\|\leq \|a^2+b^2\|,$$ for every $a,b\in \mathcal{J}$ is called a \emph{JB-algebra}. JB-algebras are precisely the self adjoint parts of JB$^*$-algebras \cite{Wright77}. A JBW$^*$-algebra is a JB$^*$-algebra which is a dual Banach space (see \cite[\S 4]{Hanche} for a detailed presentation with basic properties). \smallskip

Every real or complex associative Banach algebra is a real or complex Jordan Banach algebra with respect to the natural Jordan product $a\circ b = \frac12 (a b +ba)$.\smallskip

Let $\mathcal{J}$ be a Jordan algebra. A \emph{Jordan $\mathcal{J}$-module} is a vector space $X$, equipped with a couple of bilinear products
$(a,x)\mapsto a \circ x$ and $(x,a)\mapsto x \circ a$ from $\mathcal{J}\times X$ to $X$, satisfying: \begin{equation}\label{axioms of Jordan module 1} a \circ x = x\circ a,\ \ a^2 \circ (x \circ a) = (a^2\circ  x)\circ a, \hbox{ and, }
\end{equation}
\begin{equation}\label{axioms of Jordan module 2} 2((x\circ a)\circ  b) \circ a + x\circ (a^2 \circ b) = 2 (x\circ a)\circ  (a\circ b) + (x\circ b)\circ a^2,
 \end{equation}
for every $a,b\in \mathcal{J}$ and $x\in X$ (see  \cite[\S II.5,p.82]{Jac} for the basic facts and definitions of Jordan modules). When $X$ is a Banach space and a Jordan $\mathcal{J}$-module for which there exists $m\geq 0$ satisfying $\|a\circ x\|\leq M\ \|a\| \ \|x\|,$ we say that $X$ is a Jordan-Banach $\mathcal{J}$-module. For example, every associative Banach $A$-bimodule over a Banach algebra $A$ is a Jordan-Banach $A$-module for the product $a\circ x = \frac12 (a x + x a)$ ($a\in A$, $x\in X$). The dual, $\mathcal{J}^*$, of a Jordan Banach algebra $\mathcal{J}$ is a Jordan-Banach $J$-module with respect to the product \begin{equation}\label{eq Jordan module structure in the dual} (a\circ \varphi ) (b) = \varphi (a\circ b),
 \end{equation} where $a,b\in \mathcal{J}$, $\varphi\in \mathcal{J}^*$.\smallskip

Given a Banach $A$-bimodule $X$ over a C$^*$-algebra $A$ (respectively, a Jordan Banach $\mathcal{J}$-module over a JB$^*$-algebra $\mathcal{J}$), it is very useful to consider $X^{**}$ as a Banach $A$-bimodule or as a Banach $A^{**}$-bimodule (respectively, as a Jordan Banach $\mathcal{J}$-module or as a Jordan Banach $\mathcal{J}^{**}$-module). The case of Banach bimodules over C$^*$-algebras is very well treated in the literature (see \cite{Dales00} or \cite[\S 3]{BurFerPe2013}), we recall here the basic facts: Let $X$, $Y$ and $Z$ be Banach spaces and let $m : X\times Y \to Z$ be a bounded bilinear mapping. Defining $m^* (z^\prime,x) (y) := z^\prime (m(x,y))$ $(x\in X, y\in Y, z^\prime\in Z^*)$, we obtain a bounded bilinear mapping $m^*: Z^*\times X \to Y^*.$ Iterating the process, we define a mapping $m^{***}: X^{**}\times Y^{**} \to Z^{**}.$ The mapping $x^{\prime\prime}\mapsto  m^{***}(x^{\prime\prime} , y^{\prime\prime})$ is weak$^*$ to weak$^*$ continuous whenever we fix  $y^{\prime\prime} \in  Y^{**}$, and the mapping $y^{\prime\prime}\mapsto  m^{***}(x, y^{\prime\prime})$ is weak$^*$ to weak$^*$ continuous for every $x\in  X$. One can consider the transposed mapping $m^{t} : Y\times X\to Z,$ $m^{t} (y,x) = m(x,y)$ and the extended mapping $m^{t***t}: X^{**}\times Y^{**} \to Z^{**}.$ In this case, the mapping $x^{\prime\prime}\mapsto  m^{t***t}(x^{\prime\prime} , y)$ is weak$^*$ to weak$^*$ continuous whenever we fix  $y \in  Y$, and the mapping $y^{\prime\prime}\mapsto  m^{t***t}(x^{\prime\prime}, y^{\prime\prime})$ is weak$^*$ to weak$^*$ continuous for every $x^{\prime\prime}\in  X^{**}$.\smallskip

In general, the mappings $m^{t***t}$ and $m^{***}$ do not coincide (cf. \cite{Arens51}). When $m^{t***t}=m^{***},$ we say that $m$ is Arens regular. When $m$ is Arens regular, its (unique) third Arens transpose $m^{***}$ is separately weak$^*$ continuous (see \cite[Theorem 3.3]{Arens51}). It is well known that the product of every C$^*$-algebra $A$ is Arens regular and the unique Arens extension of the product of $A$ to $A^{**}\times A^{**}$ coincides with the product of its enveloping von Neumann algebra (cf. \cite[Corollary 3.2.37]{Dales00}). Combining \cite[Theorem 3.3]{Arens51} with \cite[Theorem 4.4.3]{Hanche}, we can deduce that the product of every JB$^*$-algebra $\mathcal{J}$ is Arens regular and the unique Arens extension of the product of $\mathcal{J}$ to $\mathcal{J}^{**}\times \mathcal{J}^{**}$ coincides with the product of $\mathcal{J}^{**}$ given by \cite[Theorem 4.4.3]{Hanche}. The literature contains some other results assuring that certain bilinear operators are Arens regular. For example, if every operator from $X$ into $Y^*$ is weakly compact and the same property holds for every operator from $Y$ into $X^*$,  then it follows from \cite[Theorem 1]{BomVi} that every bounded bilinear mapping $m : X\times Y \to Z$ is Arens regular. It is known that every bounded operator from a JB$^*$-algebra into the dual of another JB$^*$-algebra is weakly compact (cf. \cite[Corollary 3]{ChuIoLo}), thus given a JB$^*$-algebra $\mathcal{J},$ every bilinear mapping $m: \mathcal{J} \times \mathcal{J}\to Z$ is Arens regular.\smallskip

Let $X$ be a Banach $A$-bimodule over a C$^*$-algebra $A$.  Let us denote by $$\pi_1: A\times X \to X,\hbox{ and } \pi_2: X\times A \to X,$$ the bilinear maps given by the corresponding module operations, that is, $\pi_1(a,x) = a x$, and $\pi_2(x,a) = x a$, respectively. The third Arens bitransposes $\pi_1^{***}: A^{**}\times X^{**} \to X^{**}$, and $\pi_2^{***}: X^{**}\times A^{**} \to X^{**}$ satisfy that $\pi_1^{***}(a,x)$ defines a weak$^*$ to weak$^*$ linear operator whenever we fix  $x \in  X^{**}$, or whenever we fix $a\in A$, respectively, while $\pi_2^{***}(x,a)$ defines a weak$^*$ to weak$^*$ linear operator whenever we fix  $x\in  X$, and $a\in A^{**}$, respectively. From now on, given $a\in A^{**}$, $z\in X^{**},$ $b\in \mathcal{J}$ and $y\in Y^{**}$, we shall frequently  write $a z = \pi_1^{***} (a,z)$, $z a = \pi_2^{***} (z,a)$, and $b\circ y = \pi^{***} (b,y)$, respectively.
Let $(a_\lambda)$, and $(x_\mu)$ be nets in $A$ and $X$, such that $a_\lambda \to a\in A^{**}$, and  $x_\mu\to x\in X^{**}$, in the respective weak$^*$ topologies. It follows from the above properties that \begin{equation}\label{eq product bidual module} \pi_1^{***}(a,x) = \lim_{\lambda} \lim_{\mu} a_\lambda x_\mu, \hbox{ and } \pi_2^{***} (x,a) =  \lim_{\mu} \lim_{\lambda}  x_\mu a_\lambda, \end{equation} in the weak$^*$ topology of $X^{**}$. It follows from above properties that $X^{**}$ is a Banach $A^{**}$-bimodule for the above operations (cf. \cite[Theorem 2.6.15$(iii)$]{Dales00}).\smallskip

In the Jordan setting, we do not know any reference asserting that the bidual $Y^{**}$ of a Jordan Banach $\mathcal{J}$-module $Y$ over a JB$^*$-algebra $\mathcal{J}$ is a Jordan Banach $\mathcal{J}^{**}$-module, this is for the moment an open problem. However, in the particular case of $Y= J^{*}$, it is quite easy and natural to check that $J^{***}$ is a  Jordan Banach $\mathcal{J}^{**}$-module with respecto to the product defined in \eqref{eq Jordan module structure in the dual}. In nay case, there exists an alternative method, given $\varphi\in Y^*$ and $a\in \mathcal{J}$ let us define $\varphi \circ a = a\circ \varphi \in Y^*$, as the functional determined by $(\varphi \circ a) (y) := \varphi (a \circ y)$ ($y\in Y$). Iterating the process we show that $Y^{**}$ is a Jordan Banach $\mathcal{J}$-module.\smallskip

\subsection{Jordan derivations}

Let $X$ be a Jordan-Banach module over a Jordan Banach algebra $\mathcal{J}$. A \emph{Jordan derivation} from $\mathcal{J}$ into $X$ is a linear map $D: \mathcal{J}\to X$ satisfying: $$D(a\circ b ) = D(a)\circ b + a \circ D(b).$$ Following standard notation, given $x\in X$ and $a\in \mathcal{J}$, the symbols $L(a)$ and $L(x)$ will denote the mappings $L(a) : X\to X$, $x\mapsto L(a) (x) =  a\circ x$ and $L(x): \mathcal{J}\to X$, $a \mapsto L(x)(a) = a\circ x$. By a little abuse of notation, we also denote by $L(a)$ the operator on $\mathcal{J}$ defined by $L(a) (b)  = a\circ b$. Examples of Jordan derivations can be given as follows: if we fix $a\in \mathcal{J}$ and $x\in X$, the mapping $$[L(x), L(a)] = L(x) L(a) -L(a) L(x) : \mathcal{J} \to X, \ b\mapsto [L(x), L(a)] (b),$$ is a Jordan derivation. A derivation $D: \mathcal{J} \to X$ which writes in the form $D= \sum_{i=1}^{m} \left(L(x_i)L(a_i)-L(a_i)L(x_i)\right)$, ($x_i\in X, a_i\in \mathcal{J}$) is called an \emph{inner derivation}. \smallskip

In 1996, B.E. Johnson proved that every bounded Jordan derivation from a
C$^*$-algebra $A$ to a Banach $A$-bimodule is a derivation (cf. \cite{John96}).
B. Russo and the second author of this paper showed that every Jordan derivation
from a C$^*$-algebra $A$ to a Banach $A$-bimodule or to a Jordan Banach $A$-module
is continuous (cf. \cite[Corollary 17]{PeRu}). Actually every Jordan derivation from a JB$^*$-algebra $\mathcal{J}$ into $\mathcal{J}$ or into $\mathcal{J}^*$ is continuous (cf. \cite[Corollary 2.3]{HejNik96} and also \cite[Corollary 10]{PeRu}).\smallskip

Contrary to Sakai's theorem, which affirms that every derivation on a von Neumann algebra is inner \cite[Theorem 4.1.6]{Sak}, there exist examples of JBW$^*$-algebras admitting non-inner derivations (cf. \cite[Theorem 3.5 and Example 3.7]{Up80}). Following \cite{HoPerRusQJM}, a JB$^*$-algebra $\mathcal{J}$ is said to be \emph{Jordan weakly amenable}, if every (bounded) derivation from $\mathcal{J}$ into $\mathcal{J}^*$ is inner. Another difference between C$^*$-algebras and JB$^*$-algebras is that Jordan algebras do not exhibit a good behaviour concerning Jordan weak amenability; for example $L(H)$ and $K(H)$ are not Jordan weakly amenable when $H$ is an infinite dimensional complex Hilbert space (cf. \cite[Lemmas 4.1 and 4.3]{HoPerRusQJM}). Jordan weak amenability is deeply connected with the more general notion of ternary weak amenability (see \cite{HoPerRusQJM}). More interesting results on ternary weak amenability were recently developed by R. Pluta and B. Russo in \cite{PluRu}.\smallskip

Let us assume that $\mathcal{J}$ and $X$ are unital. Following \cite{BurFerGarPe2014}, a linear mapping $G: \mathcal{J}\to X$ will be called a \emph{generalised Jordan derivation} whenever $$G (a\circ b) = G(a)\circ b + a\circ G(b) - U_{a,b} G(1),$$ for every $a,b$ in $\mathcal{J}$, where $U_{a,b} (x) := (a\circ x) \circ b + (b\circ x)\circ a - (a\circ b) \circ x$ ($x\in \mathcal{J}$ or $x\in X$). Following standard notation, given an element $a$ in a JB$^*$-algebra $\mathcal{J}$, the mapping $U_{a,a}$ is usually denoted by $U_a$.  Every generalized (Jordan) derivation $G : A\to X$  with $G(1) =0$ is a Jordan derivation. Every Jordan derivation $D: \mathcal{J}\to X$ satisfies $D(1) = 0$, and hence $D$ is a generalized derivation. For each $x\in X$, the mapping $L(x) : \mathcal{J}\to X$ is a generalized derivation, and, as in the associative setting, there are examples of generalized derivations which are not derivation (cf. \cite[comments after Lemma 3]{BurFerGarPe2014}). In the not necessarily unital case, a linear mapping $G: \mathcal{J}\to X$ will be called a \emph{generalized Jordan derivation} if there exists $\xi\in X^{**}$ satisfying \begin{equation}\label{eq generalized Jordan derivation} G (a\circ b) = G(a)\circ b + a\circ G(b) - U_{a,b} (\xi ),
 \end{equation} for every $a,b$ in $\mathcal{J}$ (this definition was introduced in \cite[\S 4]{AlBreExVill09} and in \cite[\S 3]{BurFerPe2013}).\smallskip

Let $\mathcal{J}$ be a JB$^*$-algebra and let $Y$ denote $\mathcal{J}$ or $\mathcal{J}^*$, regarded as a Jordan Banach $\mathcal{J}$-module. Suppose $G: \mathcal{J} \to Y$ is a generalized derivation, and let $\xi\in Y^{**}$ denote the element for which \eqref{eq generalized Jordan derivation} holds. As we have commented before, $L(\xi) : \mathcal{J} \to Y^{**}$ is a generalized Jordan derivation. If we regard $G$ as a linear mapping from $\mathcal{J}$ into $Y^{**},$ it is not hard to check that $\widetilde{G} = G- L(\xi) : \mathcal{J}\to Y^{**}$ is a Jordan derivation. Corollary 2.3 in \cite{HejNik96} implies that $\widetilde{G}$ is continuous. If in the setting of C$^*$-algebras we replace \cite[Corollary 2.3]{HejNik96} with \cite[Corollary 17]{PeRu}, then the above arguments remain valid to obtain:

\begin{proposition}\label{p automatic cont generalized Jordan derivations} Every generalized Jordan derivation from a JB$^*$-algebra $\mathcal{J}$ into itself or into $\mathcal{J}^*$ is continuous. Furthermore, every generalized derivation from a C$^*$-algebra $A$ into a Banach $A$-bimodule is continuous. $\hfill\Box$
\end{proposition}

A consequence of the result established by T. Ho, B. Russo and the second author of this note in \cite[Proposition 2.1]{HoPerRusQJM} implies that for every Jordan derivation $D$ from a JB$^*$-algebra $\mathcal{J}$ into its dual, its bitranspose $D^{**} : \mathcal{J}^{**} \to \mathcal{J}^{***}$ is a Jordan derivation and $D^{**} (\mathcal{J}^{**}) \subseteq \mathcal{J}^{*}$. A similar technique is valid to prove the following:

\begin{proposition}\label{p bidual extensions}
Let $\mathcal{J}$ be a JB-algebra or a JB$^*$-algebra, and suppose that $G : \mathcal{J} \to \mathcal{J}^*$ is a generalized Jordan derivation (respectively, a Jordan derivation). Then $G^{**} : \mathcal{J}^{**} \to \mathcal{J}^{***}$ is a weak$^*$-continuous generalized Jordan derivation  (respectively, Jordan derivation) satisfying $G^{**} (\mathcal{J}^{**}) \subseteq \mathcal{J}^{*}$.
\end{proposition}

\begin{proof} Suppose first that $\mathcal{J}$ be a JB-algebra. It is known that $\widehat{\mathcal{J}}=\mathcal{J}+i \mathcal{J}$ can be equipped with a structure of JB$^*$-algebra such that $\widehat{\mathcal{J}}_{sa}=\mathcal{J}$ (cf. \cite{Wright77}). It is easy to check that, given a generalized Jordan derivation  $G : \mathcal{J} \to \mathcal{J}^*$, the mapping $\widehat{G} : \widehat{\mathcal{J}}\to \widehat{\mathcal{J}}^*$, $\widehat{G} (a+ i b) = G(a) + i G(b)$ ($a,b\in \mathcal{J}$) defines a generalized Jordan derivation on $\widehat{\mathcal{J}}$, where, as usually, for $\varphi\in \mathcal{J}^*$, we regard $\varphi: \widehat{\mathcal{J}} \to \mathbb{C}$ by defining $\varphi (a+i b) = \varphi (a)+ i \varphi (b)$. We may therefore assume that  $\mathcal{J}$ is a  JB$^*$-algebra.\smallskip

By Proposition \ref{p automatic cont generalized Jordan derivations}, every generalized Jordan derivation $G : \mathcal{J} \to \mathcal{J}^*$ is automatically continuous. Furthermore, since every bounded operator from a JB$^*$-algebra into the dual of another JB$^*$-algebra is weakly compact (cf. \cite[Corollary 3]{ChuIoLo}), we deduce from \cite[Lemma 2.13.1]{HiPhi} that $G^{**} (\mathcal{J}^{**}) \subset \mathcal{J}^{*}$.\smallskip

Since $G : \mathcal{J} \to \mathcal{J}^*$ is a generalized Jordan derivation, there exists $\xi\in \mathcal{J}^{***}$ satisfying $$G (x\circ y) = G(x)\circ y + x\circ G(y) - U_{x,y} (\xi ),$$ for every $x,y$ in $\mathcal{J}$. Let $a$ and $b$ be elements in $\mathcal{J}^{**}$. By Goldstine's Theorem, we can find two (bounded) nets $(a_\lambda)$ and $(b_\mu)$ in $\mathcal{J}$ such that $(a_\lambda)\to a$ and $(b_\mu)\to b$  in the weak$^*$-topology of $\mathcal{J}^{**}$. If we fix an element $c$ in $\mathcal{J}^{**}$, and we take a net $(\phi_{\lambda})$ in $\mathcal{J}^{***}$,
converging to some $\phi\in \mathcal{J}^{***}$ in the $\sigma (\mathcal{J}^{***},\mathcal{J}^{**})$-topology,
the net $({\phi_{\lambda}}\circ c)$ converges in the $\sigma (\mathcal{J}^{***},\mathcal{J}^{**})$-topology to ${\phi}\circ c$.
The weak$^*$-continuity of the mapping $G^{**}$ implies that
$$G^{**} (a \circ c) = \hbox{w$^*$-}\lim_{\lambda} G (a_{\lambda} \circ c) = \hbox{w$^*$-}\lim_{\lambda} G(a_{\lambda})\circ c + a_{\lambda} \circ G(c) - U_{a_{\lambda},c} (\xi )$$
$$= G^{**} (a)\circ c + a \circ G(c) - U_{a,c} (\xi ),$$ for every $c\in \mathcal{J}$. This shows that  $G^{**} (a \circ c) =  G^{**} (a)\circ c + a \circ G(c) - U_{a,c} (\xi ),$ for every $c\in \mathcal{J}$, $a\in \mathcal{J}^{**}$. Therefore $$G^{**} (a\circ b)= \hbox{w$^*$-}\lim_{\mu} G^{**} (a\circ b_{\mu}) = \hbox{w$^*$-}\lim_{\mu}   G^{**} (a)\circ b_{\mu} + a \circ G(b_{\mu}) - U_{a,b_{\mu}} (\xi )$$ $$=G^{**} (a)\circ b + a \circ G^{**}(b) - U_{a,b} (\xi ),$$ giving the desired conclusion.
\end{proof}

\begin{remark}\label{r element xi in a gen Jordan der} Let $G : \mathcal{J} \to \mathcal{J}^*$ be a generalized Jordan derivation, where $\mathcal{J}$ in a JB$^*$-algebra. Let $\xi\in \mathcal{J}^{***}$ satisfying $$ G (a\circ b) = G(a)\circ b + a\circ G(b) - U_{a,b} (\xi ),$$ for every $a,b$ in $\mathcal{J}$. The previous Proposition \ref{p bidual extensions} assures that $G^{**} : \mathcal{J}^{**} \to \mathcal{J}^{***}$ is a weak$^*$-continuous generalized Jordan derivation, $G^{**} (\mathcal{J}^{**}) \subseteq \mathcal{J}^{*}$, and $$ G^{**} (a\circ b) = G^{**}(a)\circ b + a\circ G^{**}(b) - U_{a,b} (\xi ),$$ for every $a,b$ in $\mathcal{J}^{**}$. In particular, $G^{**} (1) = \xi \in \mathcal{J}^*$, and $G$ is a Jordan derivation if and only if $G^{**} (1) =0$.
\end{remark}

\section{Orthogonal forms}\label{sec:orthog forms}

In the non-associative setting of JB$^*$-algebras, a Jordan version of Goldstein's theorem remains unexplored. In this section we shall study the structure of the orthogonal forms on a JB$^*$-algebra  $\mathcal{J}$. In this non-associative setting, the lacking of a Jordan version of Goldstein's theorem makes, a priori, unclear why a form on $\mathcal{J}$ which is orthogonal on $\mathcal{J}_{sa}$ is orthogonal on the whole $\mathcal{J}$. We shall prove that symmetric orthogonal forms on a JB$^*$-algebra $\mathcal{J}$ are in one to one correspondence with the \emph{purely Jordan generalized derivations} from $\mathcal{J}$ into $\mathcal{J}^{*}$ (see Theorem \ref{t one-to-one correspondence between generalized purely Jordan derivations and symmetric orthognal forms}), while anti-symmetric orthogonal forms on $\mathcal{J}$ are in one to one correspondence with the \emph{Lie Jordan derivations} from $\mathcal{J}$ into $\mathcal{J}^{*}$ (see Theorem \ref{t one-to-one correspondence between generalized purely Jordan derivations and anti-symmetric orthognal forms}).
These results, together with the existence of JB$^*$-algebras $\mathcal{J}$ which are not Jordan weakly amenable (i.e., they admit Jordan derivations from $\mathcal{J}$ into $\mathcal{J}^*$ which are not inner), shows that a Jordan version of Goldstein's theorem for anti-symmetry orthogonal forms on a JB$^*$-algebra is a hopeless task (see Remark \ref{r there is no Jordan Goldstein theorem}).\smallskip

We introduce next the exact definitions. In a JB$^*$-algebra $\mathcal{J}$ we consider the following triple product $$\{a,b,c\} = (a\circ b^*) \circ c + (c\circ b^*) \circ a - (a\circ c) \circ b^*.$$ When equipped with this triple product and its norm, every JB$^*$-algebra becomes an element in the class of JB$^*$-triples introduced by W. Kaup in \cite{Ka}. The precise definition of JB$^*$-triples reads as follows: A \emph{JB$^*$-triple} is a complex Banach space $E$ equipped with a continuous
triple product $\{\cdot,\cdot,\cdot\}:E\times E\times E\rightarrow E$ which is linear and symmetric in the outer variables, conjugate
linear in the middle one and satisfies the following conditions:
\begin{enumerate}[(JB$^*$-1)]
\item (Jordan identity) for $a,b,x,y,z$ in $E$,
$$\{a,b,\{x,y,z\}\}=\{\{a,b,x\},y,z\}
-\{x,\{b,a,y\},z\}+\{x,y,\{a,b,z\}\};$$
\item $L(a,a):E\rightarrow E$ is an hermitian
(linear) operator with non-negative spectrum, where $L(a,b)(x)=\{a,b,x\}$ with $a,b,x\in
E$;
\item $\|\{x,x,x\}\|=\|x\|^3$ for all $x\in E$.
\end{enumerate}

We refer to the monographs \cite{Hanche}, \cite{Chu2012}, and \cite{CabRod2014} for the basic background on JB$^*$-algebras and JB$^*$-triples.\smallskip

A JBW$^*$-triple is a JB$^*$-triple which is also a dual Banach space
(with a unique isometric predual \cite{BarTi}). It is known that
the triple product of a JBW$^*$-triple is separately weak*-continuous
\cite{BarTi}. The second dual of a
JB$^*$-triple $E$ is a JBW$^*$-triple with a product extending that of $E$ (compare \cite{Di86b}).\smallskip

An element $e$ in a JB$^*$-triple $E$ is said to be a
\emph{tripotent} if $\J eee =e$. Each tripotent $e$ in $E$ gives
raise to the so-called \emph{Peirce decomposition} of $E$
associated to $e$, that is,
$$E= E_{2} (e) \oplus E_{1} (e) \oplus E_0 (e),$$ where for
$i=0,1,2,$ $E_i (e)$ is the $\frac{i}{2}$ eigenspace of $L(e,e)$.
The Peirce decomposition satisfies certain rules known as
\emph{Peirce arithmetic}: $$\J {E_{i}(e)}{E_{j} (e)}{E_{k}
(e)}\subseteq E_{i-j+k} (e),$$ if $i-j+k \in \{ 0,1,2\}$ and is
zero otherwise. In addition, $$\J {E_{2} (e)}{E_{0}(e)}{E} = \J
{E_{0} (e)}{E_{2}(e)}{E} =0.$$ The corresponding \emph{Peirce
projections} are denoted by $P_{i} (e) : E\to E_{i} (e)$,
$(i=0,1,2)$. The Peirce space $E_2 (e)$ is a JB$^*$-algebra with
product $x\bullet_{e} y := \J xey$ and involution $x^{\sharp_e} :=
\J exe$.
\smallskip

For each element $x$ in a JB$^*$-triple $E$, we shall denote $x^{[1]}
:= x$, $x^{[3]} := \J xxx$, and $x^{[2n+1]} := \J xx{x^{[2n-1]}},$
$(n\in \NN)$. The symbol $E_x$ will stand for the JB$^*$-subtriple
generated by the element $x$. It is known that $E_x$ is JB$^*$-triple
isomorphic (and hence isometric) to $C_0 (\Omega)$ for some
locally compact Hausdorff space $\Omega$ contained in $(0,\|x\|],$
such that $\Omega\cup \{0\}$ is compact, where $C_0 (\Omega)$
denotes the Banach space of all complex-valued continuous
functions vanishing at $0.$ It is also known that if $\Psi$
denotes the triple isomorphism from $E_x$ onto $C_{0}(\Omega),$
then $\Psi (x) (t) = t$ $(t\in \Omega)$ (cf. Corollary
4.8 in \cite{Ka0}, Corollary 1.15 in \cite{Ka} and \cite{FriRu85}).\smallskip

Therefore, for each $x\in E$, there exists a unique element $y\in
E_x$ satisfying that $\J yyy =x$. The element $y,$ denoted by
$x^{[\frac13 ]}$, is termed the \emph{cubic root} of $x$. We can
inductively define, $x^{[\frac{1}{3^n}]} =
\left(x^{[\frac{1}{3^{n-1}}]}\right)^{[\frac 13]}$, $n\in \NN$.
The sequence $(x^{[\frac{1}{3^n}]})$ converges in the
weak*-topology of $E^{**}$ to a tripotent denoted by $r(x)$ and
called the \emph{range tripotent} of $x$. The element $r(x)$ is
the smallest tripotent $e\in E^{**}$ satisfying that $x$ is
positive in the JBW$^*$-algebra $E^{**}_{2} (e)$ (compare \cite{EdRu}, Lemma
3.3).\smallskip

Elements $a,b$ in a JB$^*$-algebra $\mathcal{J}$, or more generally, in a JB$^*$-triple $E$, are said to be \emph{orthogonal} (denoted by $a\perp b$) when $L(a,b)=0$, that is, the triple product $\{a,b,c\} $ vanishes for every $c\in \mathcal{J}$ or in $E$ (\cite{BurFerGarMarPe,BurFerGarPe}). An application of \cite[Lemma 1]{BurFerGarMarPe} assures that $a\perp b$ if and only if one of the following statements holds:
\begin{equation}
\label{ref orthogo}\begin{array}{ccc}
  \J aab =0; & a \perp r(b); & r(a) \perp r(b); \\
  & & \\
  E^{**}_2(r(a)) \perp E^{**}_2(r(b));\ \ \ & r(a) \in E^{**}_0 (r(b));\ \ \  & a \in E^{**}_0 (r(b)); \\
  & & \\
  b \in E^{**}_0 (r(a)); & E_a \perp E_b & \J bba=0.
\end{array}
\end{equation}
The above equivalences, in particular imply that the relation of being orthogonal is a ``local concept'', more precisely, $a\perp b$ in $\mathcal{J}$ (respectively in $E$) if and only if $a\perp b$ in a JB$^*$-subalgebra (respectively, JB$^*$-subtriple) $\mathcal{K}$ containing $a$ and $b$.\smallskip

Suppose $a\perp b$ in $\mathcal{J}$, applying the above arguments we can always assume that $\mathcal{J}$ is unital. In this case, $a\circ b^* = \{a,b,1\} =0$ and $(a\circ a^*) \circ b - (a\circ b) \circ a^*=(a\circ a^*) \circ b + (b\circ a^*) \circ a - (a\circ b) \circ a^*=0,$ therefore $a\circ b^* =0$ and $(a\circ a^*) \circ b = (a\circ b) \circ a^*$. Actually the last two identities also imply that $a\perp b$. It follows that \begin{equation}\label{eq orthogonality in JB*-algeb} a\perp b \Leftrightarrow a\circ b^* =0 \hbox{ and } (a\circ a^*) \circ b = (a\circ b) \circ a^*.
\end{equation} So, if $a\perp b$ and $c$ is another element in $\mathcal{J}$, we deduce, via Jordan identity, that
$$\{U_a (c), U_a (c), b\} = \{ \{a,c^*,a\}, \{a,c^*,a\}, b\} =  - \{c^*,a, \{ \{a,c^*,a\},a,b\}\} $$
$$+ \{  \{c^*,a,\{a,c^*,a\}\},a,b\} + \{ \{a,c^*,a\},a,\{c^*,a,b\}\} =0,$$ which shows that $U_a (c) \perp b$.\smallskip

We shall also make use of the following fact
\begin{equation}\label{eq implication of orthogonality} a\perp b \hbox{ in } \mathcal{J} \Rightarrow (c\circ b^*) \circ a = (a\circ c) \circ b^*,
\end{equation} for every $c\in \mathcal{J}$, this means, in the terminology of \cite{Top}, that $a$ and $b^*$ operator commute in $\mathcal{J}$.
For the proof, we observe that, since $a\perp b$, $a \circ b^*=0$, and the involution preserves triple products, we have $0=\{a,b,c\} = (a\circ b^*) \circ c + (c\circ b^*) \circ a - (a\circ c) \circ b^*,$ we obtain the desired equality. A direct application of \eqref{eq implication of orthogonality} and \eqref{eq orthogonality in JB*-algeb} shows that \begin{equation}\label{eq implication of orthogonality 2} a\perp b \hbox{ in } \mathcal{J} \Rightarrow  (a^2) \circ b^* = (a\circ b^*) \circ a =0.
\end{equation}

When a C$^*$-algebra $A$ is regarded with its structure of JB$^*$-algebra, elements $a,b$ in $A$ are orthogonal in the associative sense if and only if they are orthogonal in the Jordan sense.\smallskip

\begin{definition}\label{def orhtogoanl forms Jordan} A form $V : \mathcal{J}\times \mathcal{J} \to \mathbb{C}$ is said to be {orthogonal}
when $V (a,b^*) = 0$ for every $a,b\in \mathcal{J}$ with $a\perp b$. If $V (a,b) = 0$ only for elements $a,b\in \mathcal{J}_{sa}$ with $a\perp b$, we shall say that $V$ is orthogonal on $\mathcal{J}_{sa}$.
\end{definition}

\subsection{Purely Jordan generalized derivations and symmetric orthogonal forms}

We begin this subsection dealing with symmetric orthogonal forms on a C$^*$-algebra, a setting in which these forms have been already studied. Let $V : A\times A \to X$ be a symmetric, orthogonal form on a C$^*$-algebra. By Goldstein's theorem (cf. Theorem \cite{Gold}), there exists a unique functional $\phi_{_V}\in A^*$ satisfying that $V (a,b) = \phi_{_V} (a \circ b)$ for all $a,b\in A$. The statement also follows from the studies of orthogonally additive $n$-homogeneous polynomials on C$^*$-algebras developed in \cite{PaPeVi} and \cite[\S 3]{BurFerGarPe}.\smallskip

Given an element $a$ in the self adjoint part $\mathcal{J}_{sa}$ of a JBW$^*$-algebra $\mathcal{J}$, there exists a smallest projection $r(a)$ in $\mathcal{J}$ with the property that $r(a)\circ a = a$ We call $r(a)$ the range projection of $a$, and it is further known that $r(a)$ belongs JBW$^*$-subalgebra of $\mathcal{J}$ generated by $a$. It is easy to check that $r(a)$ coincides with the range tripotent of $a$ in $\mathcal{J}$ when the latter is seen as a JBW$^*$-triple, so, our notation is consistent with the previous definitions. \smallskip

We explore now the symmetric orthogonal forms on a JB$^*$-algebra.

\begin{proposition}\label{p Jordan symm orthog forms}
Let $V : \mathcal{J}\times \mathcal{J} \to \mathbb{C}$ be a symmetric form on a JB$^*$-algebra which is  orthogonal on $\mathcal{J}_{sa}$. Then there exists a unique $\phi\in \mathcal{J}^{*}$ satisfying $$V(a,b) = \phi (a\circ b),$$ for every $a,b\in \mathcal{J}.$
\end{proposition}

\begin{proof} We have already commented that the (unique) third Arens transpose $V^{***} :\mathcal{J}^{**}\times \mathcal{J}^{**} \to \mathbb{C}$  is separately weak$^*$-continuous (cf. Subsection \ref{subsec: Jordan modules}). Let $a$ be a self-adjoint element in $\mathcal{J}$. It is known that the JB$^*$-subalgebra $\mathcal{J}_a$ generated by $a$ is JB$^*$-isometrically isomorphic to a commutative C$^*$-algebra (cf. \cite[\S 3]{Hanche}). Since the restricted mapping $V|_{_{\mathcal{J}_a\times \mathcal{J}_a}} : \mathcal{J}_a\times \mathcal{J}_a \to \mathbb{C}$ is a symmetric orthogonal form, there exists a functional $\phi_a\in (\mathcal{J}_a)^*$ satisfying that $$V(c,d) = \phi_a (c\circ d),$$ for every $c,d \in \mathcal{J}_a$ (cf. Theorem \ref{thm Goldstein}). It follows from the weak$^*$-density of $\mathcal{J}_{a}$ in $(\mathcal{J}_a)^{**}$ together with the separate weak$^*$-continuity of $V^{***}$, and the weak$^*$-continuity of $\phi_a$, that $$V^{***} (c,d) = \phi_a (c\circ d),$$ for every $c,d \in (\mathcal{J}_a)^{**}$. Taking $c= a$ and $d= r(a)$ the range projection of $a$ we get \begin{equation}\label{eq range projection} V (a,a) = \phi_a (a\circ a) =  \phi_a (a^2\circ r(a)) = V^{***} (a^2, r(a)) = V^{***} (r(a),a^2),
\end{equation} for every $a\in \mathcal{J}_{sa}$.\smallskip

We claim that \begin{equation}\label{eq range replaced with unit} V^{***} (a, r(a)) = V^{***} (r(a),a) = V^{***} (a,1) = V^{***} (1,a),
\end{equation} for every positive $a\in \mathcal{J}_{sa}$. We may assume that $\|a\|=1$. We actually know that there is a set $L\subset [0,1]$ with $L\cup \{0\}$ compact such that $\mathcal{J}_a$ is isomorphic to the C$^*$-algebra $C_0 (L)$ of all continuous complex-valued functions on $L$ vanishing at $0$, and under this isometric identification the element $a$ is identified with the function $t\mapsto t$. Indeed, given $\varepsilon >0,$ let $p_{\varepsilon}= \chi_{_{[\varepsilon,1]}}$ denote the projection in $(\mathcal{J}_a)^{**},$ which coincides with the characteristic function of the set $[\varepsilon,1]\cap L.$ Clearly, $p_{\varepsilon}\leq r(a)$ in $\mathcal{J}^{**}.$ Suppose we have a function $g\in \mathcal{J}_a \equiv C_0 (L)$ satisfying $p_{\varepsilon}\circ g = g \geq 0$, that is, the cozero set of $g$ is inside the interval $[\varepsilon,1]$.\smallskip

Take a sequence $(h_n)\subset C_0 (L)$ defined by
$$h_{n} (t):=\left\{%
\begin{array}{ll}
    1, & \hbox{if $t\in L\cap [\varepsilon-\frac{1}{2n},1]$;} \\
    \hbox{affine}, & \hbox{if $t\in L\cap [\varepsilon-\frac{1}{n},\varepsilon-\frac{1}{2n} ]$;} \\
    0, & \hbox{if $t\in L\cap [0,\varepsilon-\frac{1}{n} ]$} \\
\end{array}%
\right.$$ for $n$ large enough ($n\geq m_0$). The sequence $(h_n)$ converges to $p_{\varepsilon}$ in the weak$^*$-topology of $(\mathcal{J}_a)^{**}$ and $1-h_n \perp p_{\varepsilon}, g$. So, $\mathcal{J}\ni U_{1-h_n} (c) \perp g$ for every $c\in \mathcal{J}$ and $n\geq m_0$. Since $1\in \mathcal{J}^{**}$, we can find, via Goldstine's theorem, a net $(c_{\gamma})\subset \mathcal{J}$ converging to $1$ in the weak$^*$ topology of $\mathcal{J}^{**}$. By hypothesis, $0=V(U_{1-h_n}(c_{\gamma}), g)$, for every $\lambda$, $n\geq m_0$. Taking limits in $\gamma$ and in $n$, it follows from the separate weak$^*$ continuity of $V^{***}$, that \begin{equation}\label{eq pvarepsilon} V^{***} (1-p_{\varepsilon},g)=0
\end{equation} for every $p_{\varepsilon}$ and $g$ as above. If we take $$g_{\varepsilon} (t):=\left\{%
\begin{array}{ll}
    t, & \hbox{if $t\in L\cap [2 \varepsilon,1]$;} \\
    \hbox{affine}, & \hbox{if $t\in L\cap [\varepsilon,2 \varepsilon]$;} \\
    0, & \hbox{if $t\in L\cap [0,\varepsilon]$}, \\
\end{array}%
\right.$$ then $0\leq g_{\varepsilon}\leq p_{\eta}$, for every $\eta \leq \varepsilon$, $\displaystyle \lim_{\varepsilon\to 0} \| g_{\varepsilon}-a\| =0$ and weak$^*$-$\displaystyle \lim_{\eta\to 0} p_{\eta} = r(a)$. Combining these facts with \eqref{eq pvarepsilon} and the separate weak$^*$-continuity of $V^{***}$, we get $V^{***} (1-r(a), a) =0$, which proves \eqref{eq range replaced with unit}.\smallskip

The identities in \eqref{eq range projection} and \eqref{eq range replaced with unit} show that $V(a,a) = V^{***} (1, a^2),$ for every $a\in \mathcal{J}_{sa}$. Let us define $\phi = V^{***} (1, .)\in A^{*}$. A polarization formula, and $V$ being symmetric imply that $V(a,b) = V^{***} (1, a\circ b) = \phi (a\circ b)$, for every $a,b\in \mathcal{J}_{sa}$, and by bilinearity $V(a,b) = \phi (a\circ b)$, for every $a,b\in \mathcal{J}$.
\end{proof}

The previous proposition is a generalization of Goldstein's theorem for symmetric orthogonal forms. It can be also regarded as a characterization of orthogonally additive 2-homogeneous polynomials on a JB$^*$-algebra $\mathcal{J}$. More concretely, according to the notation in \cite{PaPeVi}, a 2-homogeneous polynomial $P: \mathcal{J} \to \mathbb{C}$ is orthogonally additive on $\mathcal{J}_{sa}$ (i.e., $P(a+b)= P(a) + P(b)$ for every $a\perp b$ in $\mathcal{J}_{sa}$) if, and only if, there exists a unique $\phi\in \mathcal{J}^*$ satisfying $P(a) = \phi (a^2)$, for every $a\in \mathcal{J}$. This characterization constitutes an extension of \cite[Theorem 2.8]{PaPeVi} and \cite[Theorem 3.1]{BurFerGarPe} to the setting of JB$^*$-algebras.\smallskip

\begin{remark}\label{remark orthogonal symmetric is orthogonal Jordan} Let $V : \mathcal{J}\times \mathcal{J} \to \mathbb{C}$ be a symmetric form on a JB$^*$-algebra. The above Proposition \ref{p Jordan symm orthog forms} implies that $V$ is orthogonal if and only if it is orthogonal on $\mathcal{J}_{sa}$.
\end{remark}

Let $V : \mathcal{J}\times \mathcal{J} \to \mathbb{C}$ be a symmetric orthogonal form on a JB$^*$-algebra, and let $\phi_{V}$ be the unique functional in $\mathcal{J}^{*}$ given by Proposition \ref{p Jordan symm orthog forms}. If we define $G_{_V}: \mathcal{J} \to \mathcal{J}^*$, the operator given by $G_{_V} (a) = V(a,.)$, we can conclude that $G_{_V} (a) = \phi_{_V}\circ a = G_{\phi_{_V}} (a)$, and hence $G{_{_V}}: \mathcal{J}\to \mathcal{J}^*$ is a generalized Jordan  derivation and $V(a,b) = G_{_V} (a) (b)$ ($a,b\in \mathcal{J}$). Moreover, for every $a,b\in \mathcal{J}$,  $G{_{_V}} (a) (b) = V(a,b) = V(b,a) = G{_{_V}} (b)(a)$. This fact motivates the following definition:

\begin{definition}\label{def purely Jordan Gen der}{\rm Let $\mathcal{J}$ be a JB$^*$-algebra. A
\emph{purely Jordan generalized derivation} from $\mathcal{J}$ into $\mathcal{J}^*$ is a generalized Jordan derivation $G: \mathcal{J} \to \mathcal{J}^*$
satisfying $G(a) (b) = G(b) (a)$, for every $a,b\in \mathcal{J}$.}
\end{definition}

We have already seen that every symmetric orthogonal form $V$ on a JB$^*$-algebra $\mathcal{J}$ determines a pure Jordan generalized derivation $G_{V} :  \mathcal{J} \to \mathcal{J}^*$. To explore the reciprocal implication we shall prove that every generalized derivation from $\mathcal{J}$ into $\mathcal{J}^*$ defines an orthogonal on $\mathcal{J}_{sa}$.

\begin{proposition}\label{prop generalized Jordan derivations define orthogonal forms} Let $G: \mathcal{J} \to \mathcal{J}^*$ be a generalized Jordan derivation, where $\mathcal{J}$ is a JB$^*$-algebra. Then the form $V_{G} : \mathcal{J}\times \mathcal{J} \to \mathbb{C}$, $V_{G} (a,b) = G(a) (b)$ is orthogonal on $\mathcal{J}_{sa}$.
\end{proposition}

\begin{proof} Let $G: \mathcal{J} \to \mathcal{J}^*$ be a generalized Jordan derivation. By Proposition \ref{p automatic cont generalized Jordan derivations}, $G$ is continuous, and by Proposition \ref{p bidual extensions}, $G^{**}: \mathcal{J}^{**} \to \mathcal{J}^*$ is a generalized Jordan derivation too. Let $\xi$ denote $G^{**} (1)$.\smallskip

Let $p$ be a projection in $\mathcal{J}^{**}$ and let $b$ be any element in $\mathcal{J}^{**}$ such that $p\perp b$. Since $$G^{**}(p) = G^{**}(p\circ p ) = 2 p \circ G^{**}(p) + U_{p} (\xi),$$ we deduce that \begin{equation}\label{eq projection and orthogonal} G^{**}(p) (b^*) = 2 G^{**}(p) (p\circ b^*) + \xi (U_{p} (b^*)) =0.
\end{equation}

Let $a$ be a symmetric element in $\mathcal{J}^{**}$, and let $b$ be any element in $\mathcal{J}^{**}$ satisfying $a\perp b$. By \eqref{ref orthogo}, the JBW$^*$-algebra $\mathcal{J}^{**}_{<a>}$ generated by $a$ is orthogonal to $b$, that is, $c\perp b$ for every $c\in \mathcal{J}^{**}_{<a>}$. It is well known that $a$ can be approximated in norm by finite linear combinations of mutually orthogonal projections in $\mathcal{J}^{**}_{<a>}$ (cf. \cite[Proposition 4.2.3]{Hanche}). It follows from \eqref{eq projection and orthogonal}, the continuity of $G^{**}$, and the previous comments that $$ V_{_{G^{**}}} (a,b^*) = G^{**} (a)(b^*)= 0,$$ for every $a\in \mathcal{J}^{**}_{sa}$ and every $b\in \mathcal{J}^{**}$ with $a\perp b$.
 \end{proof}

Our next result follows now as a consequence of Proposition \ref{p Jordan symm orthog forms}, Remark \ref{remark orthogonal symmetric is orthogonal Jordan}, and Proposition \ref{prop generalized Jordan derivations define orthogonal forms}.

\begin{theorem}\label{t one-to-one correspondence between generalized purely Jordan derivations and symmetric orthognal forms}  Let $\mathcal{J}$ be a JB$^*$-algebra. Let $\mathcal{OF}_{s}(\mathcal{J})$ denote the Banach space of all symmetric orthogonal forms on $\mathcal{J}$, and let $\mathcal{PJGD}er(\mathcal{J},\mathcal{J}^*)$ the Banach space of all purely Jordan generalized derivations from $\mathcal{J}$ into $\mathcal{J}^*$. For each $V\in \mathcal{OF}_{s}(\mathcal{J})$ define $G_{_V} : \mathcal{J} \to \mathcal{J}^*$ in $\mathcal{PJGD}er(\mathcal{J},\mathcal{J}^*)$ given by $G_{_V} (a) (b) = V(a,b)$, and for each $G\in \mathcal{PJGD}er(\mathcal{J},\mathcal{J}^*)$ we set $V_{_G} : \mathcal{J}\times \mathcal{J} \to C$, $V_{_G} (a,b) := G(a) (b)$ $(a,b\in \mathcal{J})$.
Then the mappings
$$\mathcal{OF}_{s}(\mathcal{J}) \to \mathcal{PJGD}\hbox{er}(\mathcal{J},\mathcal{J}^*), \ \ \mathcal{PJGD}\hbox{er}(\mathcal{J},\mathcal{J}^*) \to \mathcal{OF}_{s}(\mathcal{J}),$$ $$V\mapsto G_{_V}, \ \ \ \  \ \ \ \  \ \ \ \  \ \ \ \  \ \ \ \  \ \ \ \  \ \ \ \ \ \ \ \  \ \ \ \ G\mapsto V_{_G},$$ define two isometric linear bijections and are inverses of each other.$\hfill\Box$
\end{theorem}

Actually, Proposition \ref{p Jordan symm orthog forms} gives a bit more:

\begin{corollary}\label{c purely Jordan genen der} Let $\mathcal{J}$ be a JB$^*$-algebra. Then, for every purely Jordan generalized derivation $G : \mathcal{J} \to \mathcal{J}^*$ there exists a unique $\phi\in \mathcal{J}^*$, such that $G = G_{\phi}$, that is, $G (a) = \phi\circ a$ ($a\in \mathcal{J}$).
\end{corollary}

\subsection{Derivations and anti-symmetric orthogonal forms}

We focus now our study on the anti-symmetric orthogonal forms on a JB$^*$-algebra. We motivate our study with the case of a C$^*$-algebra $A$. By Goldstein's theorem every anti-symmetric orthogonal form $V$ on $A$ writes in the form $V(a,b) = \psi ([a,b])= \psi (a b- ba)$ ($a,b\in A$), where $\psi\in A^*$ (cf. Theorem \ref{thm Goldstein}). Unfortunately, $\psi$ is not uniquely determined by $V$ (see \cite[Proposition 2.6 and comments prior to it]{Gold}). Anyway, the operator $D_{_V} : A\to A^*$, $D_{_V} (a) (b) = V(a,b) = [\psi,a] (b)$ defines a derivation from $A$ into $A^*$ and $D_{_V} (a) (b) =-  D_{_V} (b) (a)$ ($a,b\in A$). On the other hand, when $D: A \to A^*$ is a derivation, it follows from the weak amenability of $A$ (cf. \cite[Corollary 4.2]{Haa83}), that there exists $\psi\in A^*$ satisfying $D(a) = [a,\psi]$. Therefore, the form $V: A\times A \to \mathbb{C}$, $V_{_D} (a,b) = D(a)(b)$  is orthogonal and anti-symmetric. However, when $A$ is replaced with a JB$^*$-algebra, the Lie product doesn't make any sense. To avoid the gap, we shall consider Jordan derivations.\smallskip

It seems natural to ask whether the class of anti-symmetric orthogonal forms on a JB$^*$-algebra $\mathcal{J}$ is empty or not. Here is an example: let $c_1, \ldots, c_m\in \mathcal{J}$ and $\phi_1,\ldots, \phi_m\in \mathcal{J}^*$, and define $V: \mathcal{J} \times \mathcal{J} \to \mathbb{C}$, \begin{equation}\label{eq inner anti-symmetric form}
V(a,b) := \left(\sum_{i=1}^{m} \left[L(\phi_i),L(c_i)\right] (a)\right) (b)
\end{equation} $$= \left(\sum_{i=1}^{m} \left( \phi_i\circ (c_i \circ a) - c_i \circ (\phi_i \circ a)\right)\right) (b) = \sum_{i=1}^{m}  \phi_i\left(b \circ (c_i \circ a) - (c_i \circ b) \circ a)\right),$$ for every $a,b\in \mathcal{J}.$ Clearly, $V$ is an anti-symmetric form on $\mathcal{J}$.  It follows from \eqref{eq implication of orthogonality} that $V(a,b^*)=0$ for every $a\perp b$ in $\mathcal{J},$ that is, $V$ is an orthogonal form on $\mathcal{J}$. Further, the inner Jordan derivation $D: \mathcal{J} \to \mathcal{J}^*$, $D= \sum_{i=1}^{m} \left(L(\phi_i)L(a_i)-L(a_i)L(\phi_i)\right)$ satisfies $V(a,b)= D(a) (b)$ for every $a,b\in \mathcal{J}$.\smallskip

We shall see now that, like in the case of C$^*$-algebras and in the previous example, Jordan derivations from a JB$^*$-algebra $\mathcal{J}$ into its dual exhaust all the possibilities to produce an anti-symmetric orthogonal form on $\mathcal{J}$.
We begin with an strengthened version of Proposition \ref{prop generalized Jordan derivations define orthogonal forms}.

\begin{proposition}\label{prop generalized Jordan derivations define orthogonal forms on the whole JB-algebra} Let $G: \mathcal{J} \to \mathcal{J}^*$ be a generalized Jordan derivation, where $\mathcal{J}$ is a JB$^*$-algebra. Then the form $V_{G} : \mathcal{J}\times \mathcal{J} \to \mathbb{C}$, $V_{G} (a,b) = G(a) (b)$ is orthogonal {\rm(}on the whole $\mathcal{J}${\rm)}.
\end{proposition}

\begin{proof} We already know that every generalized Jordan derivation $G: \mathcal{J} \to \mathcal{J}^*$ is continuous (cf. Proposition \ref{p automatic cont generalized Jordan derivations}). By Proposition \ref{p bidual extensions}, $G^{**}: \mathcal{J}^{**} \to \mathcal{J}^*$ is a generalized Jordan derivation too. Let $\xi=G^{**} (1)$.\smallskip

Let $e$ be a tripotent in $\mathcal{J}^{**}$ and let $b$ be any element in $\mathcal{J}^{**}$ such that $e\perp b$. Since $\{e,e,e\}=2 (e\circ e^*)\circ e - e^2 \circ e^* =e$ we deduce that $$G^{**}(e) = 2 G^{**}( (e\circ e^*)\circ e) - G^{**} ( e^2 \circ e^* ) $$ $$= 2 G^{**} ( e\circ e^* ) \circ e + 2 ( e\circ e^* ) \circ G^{**} ( e) -2 U_{e\circ e^*, e} (\xi)$$
$$-  G^{**} ( e^2) \circ e^* -   e^2\circ  G^{**} (e^*) + U_{e^2, e^*} (\xi).$$

Therefore, \begin{equation}\label{eq tripotent and orthogonal} G^{**}(e) (b^*) = 2 G^{**} ( e\circ e^* ) \Big(b^* \circ e\Big)  + 2   G^{**} (e) \Big(( e\circ e^* )\circ b^*\Big)\end{equation}  $$-2 \xi \Big(( e\circ e^* ) \circ (e\circ b^*) + (( e\circ e^* )\circ b^*) \circ e - (( e\circ e^* )\circ e) \circ b^* \Big)$$
$$-  G^{**} ( e^2) \Big( e^* \circ b^*\Big) -    G^{**} (e^*) \Big(e^2\circ b^*\Big) + \xi \Big(e^2  \circ (e^*\circ b^*) + (e^2 \circ b^*) \circ e^* - (e^2 \circ e^*) \circ b^* \Big)$$ $$=\hbox{ (by \eqref{eq orthogonality in JB*-algeb}, \eqref{eq implication of orthogonality}, and \eqref{eq implication of orthogonality 2}) } = 2 G^{**} (e) \Big(( e\circ e^* )\circ b^*\Big)  -  G^{**} ( e^2) \Big( e^* \circ b^*\Big) $$ $$+ \xi \Big(e^2  \circ (e^*\circ b^*)  - (e^2 \circ e^*) \circ b^* \Big)$$
$$= 2 G^{**} (e) \Big(( e\circ e^* )\circ b^*\Big) -  2\left(e\circ G^{**} (e)\right) \Big( e^* \circ b^*\Big) +  U_{e} (\xi) (e^* \circ b^*)  $$ $$+ \xi \Big(e^2  \circ (e^*\circ b^*)  - (e^2 \circ e^*) \circ b^* \Big)$$
$$ = 2 G^{**} (e) \Big(( e\circ e^* )\circ b^*- ( b^*\circ e^* )\circ e\Big) +\xi \Big( 2 e\circ (e\circ (e^* \circ b^*)) - e^2\circ (e^* \circ b^*) \Big)$$
$$+ \xi \Big(e^2  \circ (e^*\circ b^*)  - (e^2 \circ e^*) \circ b^* \Big)$$ $$=\hbox{(by \eqref{eq implication of orthogonality}) } = \xi \Big( 2 e\circ (e\circ (e^* \circ b^*)) - (e^2 \circ e^*) \circ b^* \Big) $$ $$=\hbox{(\eqref{eq implication of orthogonality} applied twice) } = \xi \Big( 2 b^*\circ (e\circ (e^* \circ e)) - b^* \circ (e^2 \circ e^*)   \Big)$$ $$= \xi \Big(  b^* \circ\Big( 2 (e\circ (e^* \circ e)) - (e^2 \circ e^*) \Big) \Big)= \xi \Big(  b^* \circ\{e,e,e\} \Big)= \xi \Big(  b^* \circ e \Big)=0,$$ where in the last step we applied \eqref{eq orthogonality in JB*-algeb}.\smallskip

Let us take $a.b$ in $\mathcal{J}^{**}$, with $a\perp b$. The characterizations given in \eqref{ref orthogo} imply that the JBW$^*$-triple $\mathcal{J}^{**}_a$ generated by $a$ is orthogonal to $b$, that is, $c\perp b$ for every $c\in \mathcal{J}^{**}_a$. Lemma 3.11 in \cite{Ho87} guarantees that the element $a$ can be approximated in norm by finite linear combinations of mutually orthogonal projections in $\mathcal{J}^{**}_a$. Finally, the fact proved in \eqref{eq tripotent and orthogonal}, the continuity of $G^{**}$, and the previous comments imply that $ V_{_{G^{**}}} (a,b^*) = G^{**} (a)(b^*)= 0.$
 \end{proof}

We shall prove now that every anti-symmetric orthogonal form is given by a Jordan derivation.

\begin{proposition}\label{p Jordan anti-symm orthog forms}
Let $V : \mathcal{J}\times \mathcal{J} \to \mathbb{C}$ be an anti-symmetric form on a JB$^*$-algebra which is  orthogonal on $\mathcal{J}_{sa}$. Then the mapping $D_{_V}: \mathcal{J} \to \mathcal{J}^*$, $D_{_V}(a) (b) = V(a,b)$ {\rm(}$a,b\in \mathcal{J}${\rm)} is a Jordan derivation. % and $D_{_V} (a) (b) = -D_{_V} (b) (a)$, for every $a,b\in \mathcal{J}$.
\end{proposition}

Our strategy will follow some of the arguments given by U. Haagerup and N.J. Laustsen in \cite[\S 3]{HaaLaust}, the Jordan setting will require some simple adaptations and particularizations. The proof will be divided into several lemmas. The next lemma was established in  \cite[Lemma 3.3]{HaaLaust} for associative Banach algebras, however the proof, which is left to the reader, is also valid for JB$^*$-algebras.

\begin{lemma}\label{l n derivative Jordan}
Let $V : \mathcal{J}\times \mathcal{J} \to \mathbb{C}$ be a form on a JB$^*$-algebra. Suppose that $f,g: \mathbb{R} \to \mathcal{J}$ are infinitely differentiable functions at a point $t_0\in \mathbb{R}$. Then the map $t\mapsto V(f(t),g(t))$, $\mathbb{R}\to \mathbb{C}$, is infinitely differentiable at $t_0$ and its $n$'th derivative is given by $$\sum_{k=0}^{n} \left(
                                                             \begin{array}{c}
                                                               n \\
                                                               k \\
                                                             \end{array}
                                                           \right) V(f^{(k)} (t_0),g^{(n-k)} (t_0)). $$ $\hfill\Box$
\end{lemma}

The next lemma is also due to Haagerup and Laustsen, who established it for associative Banach algebras in \cite[Lemma 3.4]{HaaLaust}. The proof given in the just quoted paper remains valid in the Jordan setting, the details are included here for completeness reasons.

\begin{lemma}\label{l HaaLaust 3.4} Let $\mathcal{J}$ be a Jordan Banach algebra, let $\mathcal{U}$ be an additive subgroup of $\mathcal{J}$ whose linear span coincides with $\mathcal{J}$. Let $V : \mathcal{J}\times \mathcal{J} \to \mathbb{C}$ be an anti-symmetric form satisfying $V(a^2,a) =0$ for every $a\in \mathcal{U}$. Then the bounded linear operator $D_{_V}: \mathcal{J} \to \mathcal{J}^*$ given by $D_{_V} (a) (b) = V(a,b)$ for all $a,b\in \mathcal{J}$ is a Jordan derivation.
\end{lemma}

\begin{proof} Let us take $a,b\in \mathcal{U}$. It follows from our hypothesis that
$$D_{_V} (a^2) (b) - 2 \Big(a\circ D_{_V} (a)\Big) (b) =D_{_V} (a^2) (b) - 2 D_{_V} (a) (a\circ b)  $$ $$=  V(a^2,b) + 2 V(a\circ b, a)=  V(a^2,b) - 2 V(a, a\circ b)$$ $$= \frac{ V((a+b)^2, a+b) - V((a-b)^2, a-b)- 2 V(b^2,b)}{2} =0.$$ This implies that $D_{_V} (a^2) (b) = 2 \Big(a\circ D_{_V} (a)\Big) (b)$, for every $a,b\in \mathcal{U}$. It follows from the bilinearity and continuity of $V$, and the norm density of the linear span of $\mathcal{U}$ that $D_{_V} (a^2) = 2 a\circ D_{_V} (a)$, for every $a\in \mathcal{J}$, witnessing that $D_{_V}: \mathcal{J} \to \mathcal{J}^*$ is a Jordan derivation.
\end{proof}

We deal now with the proof of Proposition  \ref{p Jordan anti-symm orthog forms}.

\begin{proof}[Proof of Proposition \ref{p Jordan anti-symm orthog forms}] For each $a\in \mathcal{J}_{sa}$, let $B$ denote the JB$^*$-subalgebra of $\mathcal{J}$ generated by $a$. It is known that $B$ is isometrically isomorphic to a commutative C$^*$-algebra (see \cite[Theorem 3.2.2 and 3.2.3]{Hanche}). Clearly, $V|_{B\times B}: B\times B \to \mathbb{C}$ is an anti-symmetric form which is orthogonal on $B_{sa}$ (and hence orthogonal on $B$). Since $B$ is a commutative unital C$^*$-algebra, an application of Goldstein's theorem (cf. Theorem \ref{thm Goldstein}) shows that $V (x, y) =0.$ for every $x,y\in B$. In particular, $V(a^2,a) = 0$ for every $a\in \mathcal{J}_{sa}$. Lemma \ref{l HaaLaust 3.4} guarantees that $D_{_V}: \mathcal{J} \to \mathcal{J}^*$ is a Jordan derivation and $D_{_V} (a) (b) = -D_{_V} (b) (a)$, for every $a,b\in \mathcal{J}$.
\end{proof}

\begin{definition}\label{def Lie Jordan der}{\rm Let $\mathcal{J}$ be a JB$^*$-algebra. A Jordan derivation $D$ from $\mathcal{J}$ into $\mathcal{J}^*$ is said to be a \emph{Lie Jordan derivation} if $D(a) (b) =- D(b) (a)$, for every $a,b\in \mathcal{J}$.}
\end{definition}

Propositions \ref{prop generalized Jordan derivations define orthogonal forms on the whole JB-algebra} and \ref{p Jordan anti-symm orthog forms} give:

\begin{theorem}\label{t one-to-one correspondence between generalized purely Jordan derivations and anti-symmetric orthognal forms}  Let $\mathcal{J}$ be a JB$^*$-algebra. Let $\mathcal{OF}_{as}(\mathcal{J})$ denote the Banach space of all anti-symmetric orthogonal forms on $\mathcal{J}$, and let $\mathcal{L}ie\mathcal{JD}er(\mathcal{J},\mathcal{J}^*)$ the Banach space of all Lie Jordan derivations from $\mathcal{J}$ into $\mathcal{J}^*$. For each $V\in \mathcal{OF}_{as}(\mathcal{J})$ we define $D_{_V} : \mathcal{J} \to \mathcal{J}^*$ in $\mathcal{L}ie\mathcal{JD}er(\mathcal{J},\mathcal{J}^*)$ given by $D_{_V} (a) (b) = V(a,b)$, and for each $D\in \mathcal{L}ie\mathcal{JD}er(\mathcal{J},\mathcal{J}^*)$ we set $V_{_D} : \mathcal{J}\times \mathcal{J} \to C$, $V_{_D} (a,b) := D(a) (b)$ $(a,b\in \mathcal{J})$.
Then the mappings
$$\mathcal{OF}_{as}(\mathcal{J}) \to \mathcal{L}\hbox{ie}\mathcal{JD}\hbox{er}(\mathcal{J},\mathcal{J}^*), \ \ \mathcal{L}\hbox{ie}\mathcal{JD}\hbox{er}(\mathcal{J},\mathcal{J}^*) \to \mathcal{OF}_{as}(\mathcal{J}),$$ $$V\mapsto D_{_V}, \ \ \ \  \ \ \ \  \ \ \ \  \ \ \ \  \ \ \ \  \ \ \ \  \ \ \ \ \ \ \ \  \ \ \ \ D\mapsto V_{_D},$$ define two isometric linear bijections and are inverses of each other.$\hfill\Box$
\end{theorem}

Our final result subsumes the main conclusions of the last subsections.

\begin{corollary}\label{c characterization of orthognal forms} Let $V : \mathcal{J}\times \mathcal{J} \to \mathbb{C}$ be a form on a JB$^*$-algebra. The following statements are equivalent:\begin{enumerate}[$(a)$] \item $V$ is orthogonal;
\item $V$ is orthogonal on $\mathcal{J}_{sa}$;
\item There exist a (unique) purely Jordan generalized derivation $G: \mathcal{J} \to \mathcal{J}^*$ and a (unique) Lie Jordan derivation $D: \mathcal{J} \to \mathcal{J}^*$ such that $V(a,b) = G(a) (b) +  D(a) (b)$, for every $a,b\in \mathcal{J}$;
\item There exist a (unique) functional $\phi\in \mathcal{J}^*$ and a (unique) Lie Jordan derivation $D: \mathcal{J} \to \mathcal{J}^*$ such that $V(a,b) = G_{\phi} (a) (b) +  D(a) (b)$, for every $a,b\in \mathcal{J}$.
\end{enumerate}
\end{corollary}

\begin{proof} $(a)\Rightarrow (b)$ is clear. To see $(b)\Rightarrow (c)$ and $(b)\Rightarrow (d)$, we recall that every form $V : \mathcal{J}\times \mathcal{J} \to \mathbb{C}$ writes uniquely in the form $V = V_{s} + V_{as},$ where $V_{s}, V_{as}:  \mathcal{J} \to \mathcal{J}^*$ are a symmetric and an anti-symmetric form on $\mathcal{J}$, respectively. Furthermore, since $V_{s} (a,b) =\frac12 (V(a,b) + V(b,a))$ and $V_{as} (a,b) =\frac12 (V(a,b) - V(b,a))$ ($a,b\in \mathcal{J}$), we deduce that $V$ is orthogonal (on $\mathcal{J}_{sa}$) if and only if both $V_s$ and $V_{as}$ are orthogonal (on $\mathcal{J}_{sa}$).  Therefore, the desired implications follow from Theorems \ref{t one-to-one correspondence between generalized purely Jordan derivations and symmetric orthognal forms} and \ref{t one-to-one correspondence between generalized purely Jordan derivations and anti-symmetric orthognal forms}. The same theorems also prove $(c)\Rightarrow (a)$ and $(d)\Rightarrow (a)$.
\end{proof}

We shall finish this note with an observation which helps us to understand the limitations of Goldstein theorem in the Jordan setting.

\begin{remark}\label{r there is no Jordan Goldstein theorem}{\rm
Let $A$ be a C$^*$-algebra, since the anti-symmetric orthogonal forms on $A$ and the Lie Jordan derivations from $A$ into $A^*$ are mutually determined, we can deduce, via Goldstein's theorem (cf. Theorem \ref{thm Goldstein}), that every Lie Jordan derivation $D: A \to A^*$ is an inner derivation, i.e., a derivation given by a functional $\psi\in A^*$, that is, $D (a) = \hbox{adj}_{\psi} (a) = \psi a - a \psi$ ($a\in A$). We shall see that a finite number of functionals in the dual of a JB$^*$-algebra $\mathcal{J}$ and a finite collection of elements in $\mathcal{J},$ i.e., the inner Jordan derivations are not enough to determine the Lie Jordan derivations from $\mathcal{J}$ into $\mathcal{J}^*$ nor the anti-symmetric orthogonal forms on $\mathcal{J}$. Indeed, as we have commented before, there exist examples of JB$^*$-algebras which are not Jordan weakly amenable, that is the case of $L(H)$ and $K(H)$ when $H$ is an infinite dimensional complex Hilbert space (cf. \cite[Lemmas 4.1 and 4.3]{HoPerRusQJM}). Actually, let $B= K(H)$ denote the ideal of all compact operators on $H$, and let $\psi$ be an element in $B^*$ whose trace is not zero. The proof of \cite[Lemmas 4.1]{HoPerRusQJM} shows that the derivation $D=\hbox{adj}_\psi : B \to B^*$, $a\mapsto \psi a - a \psi$ is not inner in the Jordan sense. Therefore the anti-symmetric form $V(a,b) = D(a) (b) = (\psi a - a \psi) (b) = \psi [a,b]$ cannot be represented in the form given in \eqref{eq inner anti-symmetric form}. A similar example holds for $B=B(H)$ (cf. \cite[Lemma 4.3]{HoPerRusQJM}).
}\end{remark}

\begin{remark}\label{remark approximation}{\rm We have already shown the existence of JBW$^*$-algebras which are not Jordan weakly amenable (cf. \cite[Lemmas 4.1 and 4.3]{HoPerRusQJM}). Thus, the problem of determining whether in a JB$^*$-algebra $\mathcal{J}$, the inner Jordan derivations on $\mathcal{J}$ are norm-dense in the set of all Jordan derivations on $\mathcal{J}$, takes on a new importance. When the problem has an affirmative answer for a JB$^*$-algebra $\mathcal{J}$, Theorem \ref{t one-to-one correspondence between generalized purely Jordan derivations and anti-symmetric orthognal forms} allows us to approximate anti-symmetric orthogonal forms on $\mathcal{J}$ by a finite collection of functionals in $\mathcal{J}^*$ and a finite number of elements in $\mathcal{J}$. Related to this problem, we remark that Pluta and Russo recently proved that the set of inner triple derivations from a von Neumann algebra $M$ into its predual is norm dense in the real vector space of all triple derivations, then $M$ must be finite, and the reciprocal statement holds if $M$ acts on a separable Hilbert space, or is a factor \cite[Theorem 1]{PluRu}. It would be interesting to explore the connections between normal orthogonal forms and normal Jordan weak amenability or norm approximation by normal inner derivations.}
\end{remark}

\medskip\medskip

\end{document}